\documentclass[11pt]{amsart}
\usepackage{amssymb, amsmath, amsthm, amscd}
\newtheorem{proposition}{Proposition}[section]

\newtheorem{theorem}[proposition]{Theorem}

\theoremstyle{definition}

\newtheorem{example}[proposition]{Example}

\newtheorem{remark}[proposition]{Remark}

\newcommand{\thlabel}[1]{\label{th:#1}}
\newcommand{\thref}[1]{Theorem~\ref{th:#1}}
\newcommand{\selabel}[1]{\label{se:#1}}

\newcommand{\relabel}[1]{\label{re:#1}}
\newcommand{\reref}[1]{Remark~\ref{re:#1}}

\newcommand{\eqlabel}[1]{\label{eq:#1}}
\newcommand{\equref}[1]{(\ref{eq:#1})}

\def\ot{\otimes}

\newcommand{\Cc}{\mathcal{C}}

\def\*C{{}^*\hspace*{-1pt}{\Cc}}
\def\text#1{{\rm {\rm #1}}}

\input xy
\xyoption {all} \CompileMatrices

\begin{document}

\title[Limits of coalgebras, bialgebras and Hopf algebras]{Limits of coalgebras, bialgebras and Hopf algebras}
\dedicatory{Dedicated to the memory of Professor S. Ianu\c s}

\author{A.L. Agore}\thanks{The author acknowledges partial support
from CNCSIS grant 24/28.09.07 of PN II "Groups, quantum groups,
corings and representation theory". }
\address{Department of Mathematics, Academy of Economic Studies,
Piata Romana 6, RO-010374 Bucharest 1,  Romania}
\email{ana.agore@fmi.unibuc.ro}



\begin{abstract}
We give the explicit construction of the product of an arbitrary
family of coalgebras, bialgebras and Hopf algebras: it turns out
that the product of an arbitrary family of coalgebras (resp.
bialgebras, Hopf algebras) is the sum of a family of coalgebras
(resp. bialgebras, Hopf algebras). The equalizers of two morphisms
of coalgebras (resp. bialgebras, Hopf algebras) are also described
explicitly. As a consequence the categories of coalgebras,
bialgebras and Hopf algebras are shown to be complete and a
complete description for limits in the above categories is given.
\end{abstract}

\maketitle

\section*{Introduction}
It is well known that the category $k$-Alg of $k$-algebras is
complete and cocomplete: that is any functor $F: I \to k$-Alg has
a limit and a colimit, for all small categories $I$. This is
immediately implied by the existence of products, coproducts,
equalizers and coequalizers in the category $k$-Alg. The
categories of coalgebras, bialgebras or Hopf algebras have
arbitrary coproducts and coequalizers (see \cite[Propositon
1.4.19]{DNR}, \cite[Proposition 2.10]{BW}, \cite[Corollary
2.6.6]{P} and \cite[Remark 2.1, Theorem 2.2]{ag}), hence these
categories are cocomplete. Related to the question of whether
these categories are complete (i.e. if they have arbitrary
products and equalizers) we could not find similar results in the
classical Hopf algebra textbooks (\cite{abe}, \cite{Sw}), not even
in the more recent ones (\cite{BW}, \cite{DNR}). For example,
\cite[Propositon 1.4.21]{DNR} proves only the existence of finite
products (namely the tensor product of coalgebras) and only in the
category of \textit{cocommutative} coalgebras, as a dual result to
the one concerning commutative algebras. Only recently in
\cite[Theorem 9]{HP2} it is proved that the category of coalgebras
or, more generally, the category of corings is locally
presentable, thus they are complete by the definition of locally
presentable categories. However the proof of \cite[Theorem 9]{HP2}
does not construct explicitly the limits (in particular the
products) of an arbitrary family of coalgebras.

In this note we shall fill this gap: using the fact that the
forgetful functor from the category of coalgebras to the category
of vector spaces has a right adjoint, namely the so called cofree
coalgebra, we shall construct explicitly the product of an
arbitrary family of coalgebras. As a consequence, the product of
an arbitrary family of bialgebras and Hopf algebras is
constructed. The equalizers of two morphisms of coalgebras
(bialgebras, Hopf algebras) are also described explicitly. Thus we
shall obtain that the categories of coalgebras, bialgebras and
Hopf algebras are complete.

Throughout this paper, $k$ will be a field. Unless specified
otherwise, all vector spaces, algebras, coalgebras, bialgebras,
tensor products and homomorphisms are over $k$. Our notation for
the standard categories is as follows: ${}_k{\mathcal {M}}$
($k$-vector spaces), $k$-Alg (associative unital $k$-algebras),
$k$-CoAlg (coalgebras over $k$), $k$-BiAlg (bialgebras over $k$),
$k$-HopfAlg (Hopf algebras over $k$), ${\mathcal {M}}^{C}$ (right
$C$-comodules). For a coalgebra $C$, we will use Sweedler's
$\Sigma$-notation, that is, $\Delta(c)= c_{(1)}\ot c_{(2)},~
(I\ot\Delta)\Delta(c)= c_{(1)}\ot c_{(2)}\ot c_{(3)}$, etc. Given
a vector space $V$, $(K(V), p)$ stands for the cofree coalgebra on
$V$, where $K(V)$ is a coalgebra and $p:K(V) \rightarrow V$ is a
$k$-linear map. We refer to \cite{abe}, \cite{DNR}, \cite{Sw} for
further details concerning Hopf algebras. A category $\mathcal{C}$
is called \textit{(co)complete} if any functor $F: I \rightarrow
\mathcal{C}$ has (co)limits, where $I$ is a small category. A
category $\mathcal{C}$ is \textit{(co)complete} if and only if
$\mathcal{C}$ has (co)equalizers of all pairs of arrows and all
(co)products \cite[Theorem 6.10]{par}. Given a morphism $f \in
\mathcal{C}$ we denote by $dom (f)$ and $cod (f)$ the domain,
respectively the codomain of $f$. If $\mathcal{C}$ is a small
category we denote by $Hom(\mathcal{C})$ the set of all morphisms
of $\mathcal{C}$.

\section{Limits for coalgebras, bialgebras and
Hopf algebras}\selabel{2}

First, we explicitly construct the product of an arbitrary family
of coalgebras.

\begin{theorem}\thlabel{1}
The category $k$-CoAlg of coalgebras has arbitrary products and
equalizers. In particular, the category $k$-CoAlg of coalgebras is
complete.
\end{theorem}

\begin{proof}
Let $f$, $g : C \rightarrow D$ be two coalgebra maps and $S := \{c
\in C \, | \, f(c) = g(c) \, \}$, which is a $k$-subspace of $C$.
Let $E$ be the sum of all subcoalgebras of $C$ included in $S$.
Note that the family of subcoalgebras of $C$ included in $S$ is
not empty since it contains the null coalgebra. It is immediate
that $E$ is a subcoalgebra of $C$. We shall prove that $(E,i)$ is
the equalizer of the pair $(f,g)$ in $k$-CoAlg, where $i: E
\rightarrow C$ is the canonical inclusion. Let $E'$ be a coalgebra
and $h: E' \rightarrow C$ a coalgebra map such that $f \circ h = g
\circ h$. Then $f\bigl(h(x)\bigl) = g\bigl(h(x)\bigl)$, for all $x
\in E'$, hence $h(E') \in S$ and since $h(E')$ is a subcoalgebra
in $C$ we obtain $h(E') \subseteq E$. Thus there exists a unique
coalgebra map $h:E' \rightarrow E$ such that $i \circ h = h$.
Hence $(E,i)$ is the equalizer of the pair $(f,g)$ in the category
$k$-CoAlg of coalgebras.

Now let $(C_{i})_{i \in I}$ be a family of coalgebras and
$\bigl(\prod_{i \in I}C_{i}, (\pi_{i})_{i \in I}\bigl)$ be the
product of the $k$-modules $(C_{i})_{i \in I}$. Let
$\bigl(K(\prod_{i \in I}C_{i}), p\bigl)$ be the cofree coalgebra
over the vector space $\prod_{i \in I}C_{i}$.

Let $D$ be the sum of all subcoalgebras $E$ of $K\bigl(\prod_{i
\in I}C_{i}\bigl)$ such that $\pi_{i} \circ p \circ j_{E}$ is a
coalgebra map for all $i \in I$, where $j_{E}: E \rightarrow
K(\prod_{i \in I}C_{i})$ is the canonical inclusion. The family of
subcoalgebras of $E$ satisfying this property is nonempty since it
contains the null coalgebra. The $k$-linear map $\pi_{i} \circ p
\circ j : D \rightarrow C_{i}$ is a coalgebra map for all $i \in
I$ where $j: D \rightarrow K(\prod_{i \in I}C_{i})$ is the
canonical inclusion. We shall prove that $\bigl(D, (\pi_{i} \circ
p \circ j)_{i \in I}\bigl)$ is the product of the family of
coalgebras $(C_{i})_{i \in I}$ in $k$-CoAlg.

Indeed, let $D'$ be a coalgebra and $g_{i}: D' \rightarrow C_{i}$,
$i \in I$, a family of coalgebra maps. Using the universal
property of the product in ${}_k{\mathcal {M}}$ we obtain that
there exists a unique $k-$linear map $\overline{g}: D' \rightarrow
\prod_{i \in I}C_{i}$ such that $\pi_{i} \circ \overline{g} =
g_{i}$, for all $i \in I$. Furthermore, since $\bigl(K(\prod_{i
\in I}C_{i}), p\bigl)$ is the cofree coalgebra over the $k-$module
$\prod_{i \in I}C_{i}$, there exists a unique coalgebra map
$\overline{f}: D' \rightarrow K(\prod_{i \in I}C_{i})$ such that
$p \circ \overline{f} = \overline{g}$. Thus we have the following
commutative diagram:
$$
\xymatrix{ D' \ar[dr]_{\overline{f}}
\ar[drr]_{\overline{g}} \ar[drrr]^{g_{i}}\\
D\ar[r]_{j} & K\bigl(\prod_{i \in I}C_{i}\bigl) \ar[r]_{p} &
{\prod_{i \in I}C_{i}} \ar[r]_{\pi_{i}} & C_{i}}
$$
It follows that $(\pi_{i} \circ p)(\overline{f}(x)) = g_{i}(x)$,
for all $x \in D'$ and $i \in I$. So, since each $g_{i}$ is a
coalgebra map, we have $\overline{f}(D') \subseteq D$. Hence, we
proved that for any coalgebra $D'$ and any family $g_{i}: D'
\rightarrow C_{i}$, $i \in I$, of coalgebra maps there exists a
coalgebra map $\overline{f} : D' \rightarrow D$ such that
$(\pi_{i} \circ p \circ j) \circ \overline{f} = g_{i}$, for all $i
\in I$. Let $h: D'\rightarrow D$ be another coalgebra map such
that $(\pi_{i} \circ p \circ j) \circ h = g_{i}$ for all $i \in
I$. From the uniqueness of $\overline{g}$ we obtain $p \circ j
\circ h = \overline{g}$. Moreover, from the uniqueness of
$\overline{f}$ we obtain $j \circ h = \overline{f}$, hence $h =
\overline{f}$. Thus $\bigl(D, (\pi_{i} \circ p \circ j)_{i \in
I}\bigl)$ is the product of the family $(C_{i})_{i \in I}$ in the
category $k$-CoAlg of coalgebras.
\end{proof}

\begin{remark}\relabel{1}
1) In \cite[Lemma 1.1.3]{AD} a description for the equalizers in
the category $k$-HopfAlg is given. We can use the same method in
order to obtain another description for the equalizer of a pair of
coalgebra (or bialgebra) maps. Let $f$, $g : C \rightarrow D$ be
two coalgebra maps. It can be easily proved that $(E, i)$ is the
equalizer of the pair $(f,g)$ in the category $k$-CoAlg of
coalgebras, where
$$
E = \{ c\in C \, | \, c_{(1)} \ot f (c_{(2)}) \ot c_{(3)} =
c_{(1)} \ot g (c_{(2)}) \ot c_{(3)} \, \}
$$
and $i : E \to C$ is the canonical inclusion. This equivalent
description of equalizers in the category $k$-CoAlg will turn out
to be more efficient for computations.

\end{remark}

\begin{example}
Let $G$ be a multiplicative group and $kG$ the $k$-vector space
with basis $\{g | g \in G\}$ endowed with the classical coalgebra
structure : $\Delta(g) = g \otimes g$ and $\varepsilon(g) = 1$ for
all $g \in G$. Thus any element $x \in kG$ has the form $x =
\sum_{g \in G} k_{g}g$ where $(k_{g})_{g \in G}$ is a family of
elements in $k$ with only a finite number of non-zero elements. We
use the following formal notation $x^{-1} := \sum_{g \in G}
k_{g}g^{-1}$ and $0^{-1} = 0$.\\
Consider the coalgebra maps $f = Id_{kG}$ and $h: kG \rightarrow
kG$ given by $h(g) = g^{-1}$ for all $g \in G$. Then, in the light
of the above remark, it follows that the equalizer of the
morphisms $(f,g)$ is given by the pair $(E,i)$ where $E = \{x \in
kG | x \otimes x \otimes x = x \otimes x^{-1} \otimes x\}$ and $i$
is the canonical inclusion.
\end{example}

As an easy consequence of \cite[Chapter 5 \S2, Theorem 1]{ML} we
obtain the following description for limits in the category
$k$-CoAlg of coalgebras:

\begin{remark}
Let $J$ be a small category, $F: J \rightarrow$$k$-CoAlg be a
functor, $\bigl(\Pi_{j \in J} F(j),$ $(p_{j})_{j \in J} \bigl)$,
$\bigl(\Pi_{u \in Hom(J)} F(cod (u)), (p_{u})_{u \in Hom(J)}
\bigl)$ be the product in $k$-CoAlg of the families
$\bigl(F(j)\bigl)_{j \in J}$, respectively $\bigl(F(cod
(u))\bigl)_{u \in Hom(J)}$ and $f, g : \Pi_{j \in J} F(j)
\rightarrow \Pi_{u \in Hom(J)} F(cod (u))$ be the unique coalgebra
maps such that $p_{u} \circ f = p_{cod (u)}$ and $p_{u} \circ g =
F(u) \circ p_{dom (u)}$ for all $u \in Hom(J)$. We define
$$D = \{ x \in \Pi_{j \in J} F(j) | x_{(1)} \ot f (x_{(2)}) \ot
x_{(3)} = x_{(1)} \ot g (x_{(2)}) \ot x_{(3)} \}.$$ Then the pair
$\bigl(D, (\varphi_{j} = p_{j} \circ e)_{j \in J}\bigl)$ is the
limit of the functor $F$, where $e: D \rightarrow \Pi_{j \in J}
F(j)$ is the canonical inclusion.
\end{remark}

In what follows we will make use of \thref{1} in order to
construct the product in the category of $k$-BiAlg of bialgebras.

\begin{theorem}\thlabel{2}
The category $k$-BiAlg of bialgebras has arbitrary products and
equalizers. In particular, the category $k$-BiAlg of bialgebras is
complete.
\end{theorem}

\begin{proof}
Let $\bigl(B_{i}, m_{i}, \eta_{i}, \Delta_{i},
\varepsilon_{i}\bigl)_{i \in I}$ be a family of bialgebras and
$\bigl((\prod_{i \in I}B_{i}, \Delta, \varepsilon), (\pi_{i})_{i
\in I}\bigl)$ be the product of this family in the category
$k$-CoAlg of coalgebras. Since $(B_{i}, m_{i}, \eta_{i},
\Delta_{i}, \varepsilon_{i})$ is a bialgebra it follows that
$m_{i}: B_{i} \ot B_{i} \rightarrow B_{i}$ and $\eta_{i}:k
\rightarrow B_{i}$ are coalgebra maps for all $i \in I$.
\newline Then there exists a unique coalgebra map $\eta : k \rightarrow \prod_{i \in
I}B_{i}$ such that the following diagram :
\begin{equation}\eqlabel{4}
\xymatrix {& k\ar[d]_{\eta}\ar[dr]^{\eta_{i}}\\
 & {\prod_{i \in
I}B_{i}} \ar[r]^{\pi_{i}} & B_{i}}\end{equation} is commutative
for all $i \in I$. Also there exists a unique coalgebra map $m:
\prod_{i \in I}B_{i} \ot \prod_{i \in I}B_{i} \rightarrow \prod_{i
\in I}B_{i}$ for which the diagram :
\begin{equation}\eqlabel{5}
\xymatrix {& {\prod_{i \in I}B_{i}} \otimes {\prod_{i \in
I}B_{i}} \ar[d]_{m}\ar[dr]^{m_{i}\circ(\pi_{i} \otimes \pi_{i})}\\
&{\prod_{i \in I}B_{i}} \ar[r]^{\pi_{i}} & B_{i}}
\end{equation}
is commutative for all $i \in I$.
\newline First, we will prove that $(\prod_{i \in I}B_{i}, m, \eta)$ is a
$k$-algebra. Since $\pi_{i} \circ m \circ (m \ot Id)$ is a
coalgebra map by the universal property of the product we obtain
that there exists a unique coalgebra map $\psi : \bigl(\prod_{i
\in I}B_{i}\bigl)^{\ot 3} \rightarrow \prod_{i \in I}B_{i}$ such
that the following diagram:
$$
\xymatrix {& {\bigl(\prod_{i \in I}B_{i}\bigl)}^{\otimes 3}
\ar[d]_{\psi}\ar[dr]^{\pi_{i} \circ m\circ(m \otimes Id)}\\
 & {\prod_{i \in I}B_{i}} \ar[r]^{\pi_{i}} & B_{i}}
$$
is commutative for all $i \in I$. It is easy to see that the
coalgebra map $m \circ (m \ot Id)$ makes the above diagram
commute. Thus, using the uniqueness of $\psi$, in order to prove
that $m \circ (m \ot Id) = m \circ (Id \ot m)$ it is enough to
show that $\pi_{i} \circ m \circ (m \ot Id) = \pi_{i} \circ m
\circ (Id \ot m)$ for all $i \in I$. We have :
\begin{eqnarray*}
\pi_{i} \circ m \circ (m \ot Id) &\stackrel{\equref{5}} {=}& m_{i}
\circ (\pi_{i} \ot \pi_{i}) \circ (m \ot Id)\\
&{=}& m_{i} \circ \bigl((\pi_{i} \circ m) \ot \pi_{i}\bigl)\\
&\stackrel{\equref{5}} {=}&
m_{i} \circ \bigl[\big(m_{i} \circ (\pi_{i} \ot \pi_{i})\bigl) \ot \pi_{i}\bigl]\\
&{=}& m_{i} \circ (m_{i} \ot Id) \circ (\pi_{i} \ot \pi_{i} \ot
\pi_{i})\\
&{=}& m_{i} \circ (Id \ot m_{i}) \circ (\pi_{i} \ot \pi_{i} \ot
\pi_{i})\\
&{=}& m_{i} \circ \bigl[\pi_{i} \ot \bigl(m_{i} \circ (\pi_{i}
 \ot \pi_{i})\bigl)\bigl]\\
&\stackrel{\equref{5}} {=}& m_{i} \circ \bigl(\pi_{i} \ot (\pi_{i}
\circ m) \bigl)\\
&{=}& m_{i} \circ (\pi_{i} \ot \pi_{i}) \circ (Id \ot m)\\
&\stackrel{\equref{5}} {=}& \pi_{i} \circ m \circ (Id \ot m)
\end{eqnarray*}
Hence $m \circ (m \ot Id) = m \circ (Id \ot m)$, i.e. $m$ is
associative.

Consider now the coalgebra map $\pi_{i} \circ m \circ (\eta \ot
Id)$. From the universal property of the product, we obtain that
there exists a unique coalgebra map $\varphi: k \ot \prod_{i \in
I}B_{i} \rightarrow B_{i}$ such that the following diagram :
$$
\xymatrix {& k \otimes {\prod_{i \in I}B_{i}}
\ar[d]_{\varphi}\ar[dr]^{\pi_{i} \circ m\circ(\eta \otimes Id)}\\
& {\prod_{i \in I}B_{i}} \ar[r]^{\pi_{i}} & B_{i}}
$$
is commutative for all $i \in I$. By the argument above, in order
to prove that $m \circ (\eta \ot Id) = s$ it will be enough to
show that $\pi_{i} \circ m \circ (\eta \ot Id) = \pi_{i} \circ s$,
where we denote by $s$ the scalar multiplication. We have:
\begin{eqnarray*}
\pi_{i} \circ m \circ (\eta \ot Id)&\stackrel{\equref{5}} {=}&
m_{i} \circ (\pi_{i} \ot \pi_{i}) \circ (\eta \ot Id)\\
&{=}& m_{i} \circ \bigl((\pi_{i} \circ \eta) \ot \pi_{i}\bigl)\\
&\stackrel{\equref{4}} {=}& m_{i} \circ (\eta_{i} \ot \pi_{i})\\
&{=}& m_{i} \circ (\eta_{i} \ot Id) \circ (Id \ot \pi_{i})\\
&{=}& s \circ (Id \ot \pi_{i})
\end{eqnarray*}
Furthermore, let $k_{1} \ot b \in k\ot \prod_{i \in I}B_{i}$.
Having in mind that $\pi_{i}$ is a $k$-linear map we obtain :
\begin{eqnarray*}
s \circ (Id \ot \pi_{i})(k_{1} \ot b) &{=}& k_{1} \pi_{i}(b)\\
&{=}& \pi_{i}(k_{1}b)\\
&{=}& \pi_{i} \circ s(k_{1} \ot b)
\end{eqnarray*}
Thus we proved that $\pi_{i} \circ m \circ (\eta \ot Id) = \pi_{i}
\circ s$. In the same way it follows that $m \circ (Id \ot \eta) =
s$. Hence $(\prod_{i \in I}B_{i}, m, \eta)$ is an algebra and
since $m$ and $\eta$ are coalgebra maps we obtain that $(\prod_{i
\in I}B_{i}, m, \eta, \Delta, \varepsilon)$ is actually a
bialgebra.

To end the proof we still need to show that $(\prod_{i \in
I}B_{i}, m, \eta, \Delta, \varepsilon)$ is the product of the
family $\bigl(B_{i}, m_{i}, \eta_{i}, \Delta_{i},
\varepsilon_{i}\bigl)_{i \in I}$ in the category $k$-BiAlg. Let
$(B, m_{B}, \eta_{B}, \Delta_{B}, \varepsilon_{B})$ be a bialgebra
and $(g_{i})_{i \in I}$ be a family of bialgebra maps, $g_{i}:B
\rightarrow B_{i}$ for all $i \in I$. From the universal property
of the product, we obtain that there exists an unique coalgebra
map $\theta: B \rightarrow \prod_{i \in I}B_{i}$ such that the
following diagram commutes :
\begin{equation}\eqlabel{6}
\xymatrix {& B\ar[d]_{\theta}\ar[dr]^{g_{i}}\\
& {\prod_{i \in I}B_{i}} \ar[r]^{\pi_{i}} & B_{i}}
\end{equation}
We only need to prove that $\theta$ is also an algebra map. By the
argument used above, it is enough to show that:
\begin{equation}\eqlabel{7}
\pi_{i} \circ \theta \circ m_{B} = \pi_{i} \circ m \circ (\theta
\ot \theta) \quad {\rm ~and~} \quad \pi_{i} \circ \theta \circ
\eta_{B} =\pi_{i} \circ \eta
\end{equation}
Having in mind that $g_{i}$ is an algebra map, we have:
\begin{eqnarray*}
\pi_{i} \circ m \circ (\theta \ot \theta)&\stackrel{\equref{5}}
{=}& m_{i} \circ (\pi_{i} \ot \pi_{i}) \circ (\theta \ot \theta)\\
&{=}& m_{i} \circ \bigl((\pi_{i} \circ \theta) \ot (\pi_{i} \circ
\theta)\bigl)\\
&\stackrel{\equref{6}} {=}& m_{i} \circ (g_{i} \ot g_{i})\\
&{=}& g_{i} \circ m_{B}\\
&\stackrel{\equref{6}} {=}&\pi_{i} \circ \theta \circ m_{B}
\end{eqnarray*}
Moreover, $\pi_{i} \circ \theta \circ \eta_{B}
\stackrel{\equref{6}} {=} g_{i} \circ \eta_{B} = \eta_{i}
\stackrel{\equref{4}} {=} \pi_{i} \circ \eta$ hence \equref{7}
holds.

In what follows we construct equalizers. Let
$(A,m_{A},\eta_{A},\Delta_{A},\varepsilon_{A})$, $(B,m_{B},
\eta_{B},\Delta_{B},\varepsilon_{B})$ be two bialgebras and $f$,
$g: B \rightarrow A$ be two bialgebra maps. We denote by $S:=\{b
\in B | f(b)=g(b)\}$. Let $D$ be the sum of all subcoalgebras of
$B$ contained in $S$. We already noticed before that the family of
subcoalgebras of $B$ with this property is nonempty and that $D$
is a subcoalgebra of $B$. The pair $(D, i)$ is the equalizer of
the morphisms $(f,g)$ in $k$-BiAlg, where $i:D \rightarrow B$ is
the canonical inclusion. We only need to prove that $D$ is
actually a subbialgebra of $B$. Consider $q = \sum_{k=1}^{n}
d_{i_{k}} \ot d_{j_{k}} \in D \ot D$. We then have:
\begin{eqnarray*}
m_{A} \circ (f \ot f)(q) &{=}& m_{A} \bigl(\sum_{k=1}^{n}
f(d_{i_{k}}) \ot f(d_{j_{k}})\bigl)\\
&{=}&m_{A} \bigl(\sum_{k=1}^{n} g(d_{i_{k}}) \ot
g(d_{j_{k}})\bigl)\\
&{=}& m_{A} \circ (g \ot g)(q)
\end{eqnarray*}
Now having in mind that $f$ and $g$ are algebra maps we obtain
$f\bigl(m_{B}(D \ot D)\bigl) = g\bigl(m_{B}(D \ot D)\bigl)$, hence
$m_{B}(D \ot D) \subseteq S$ and since $m_{B}(D \ot D)$ is a
subcoalgebra it follows that $m_{B}(D \ot D) \subseteq D$. Thus
$D$ is a subbialgebra of $B$ and it can be shown as in \thref{1}
that the pair $(D, i)$ is the equalizer of the morphisms $(f, g)$
in $k$-BiAlg.
\end{proof}

As remarked before, we can obtain a description for the equalizers
in $k$-BiAlg similar to the one in \reref{1}. Thus, we have the
following description for limits in $k$-BiAlg:

\begin{remark}
Let $J$ be a small category, $F: J \rightarrow$$k$-BiAlg be a
functor, $\bigl(\Pi_{j \in J} F(j),$ $(p_{j})_{j \in J} \bigl)$,
$\bigl(\Pi_{u \in Hom(J)} F(cod (u)), (p_{u})_{u \in Hom(J)}
\bigl)$ be the product in $k$-BiAlg of the families
$\bigl(F(j)\bigl)_{j \in J}$, respectively $\bigl(F(cod
(u))\bigl)_{u \in Hom(J)}$ and $f, g : \Pi_{j \in J} F(j)
\rightarrow \Pi_{u \in Hom(J)} F(cod (u))$ be the unique bialgebra
maps such that $p_{u} \circ f = p_{cod (u)}$ and $p_{u} \circ g =
F(u) \circ p_{dom (u)}$ for all $u \in Hom(J)$. We define
$$D = \{ x \in \Pi_{j \in J} F(j) | x_{(1)} \ot f (x_{(2)}) \ot
x_{(3)} = x_{(1)} \ot g (x_{(2)}) \ot x_{(3)} \}.$$ Then the pair
$\bigl(D, (\varphi_{j} = p_{j} \circ e)_{j \in J}\bigl)$ is the
limit of the functor $F$, where $e: D \rightarrow \Pi_{j \in J}
F(j)$ is the canonical inclusion.
\end{remark}

\begin{theorem}\thlabel{4}
The category $k$-HopfAlg of Hopf algebras has arbitrary products
and equalizers. In particular, the category $k$-HopfAlg of Hopf
algebras is complete.
\end{theorem}

\begin{proof}
Let $\bigl(H_{i}, m_{i}, \eta_{i}, \Delta_{i}, \varepsilon_{i},
S_{i} \bigl)_{i \in I}$ be a family of Hopf algebras and $\bigl((B
:= \prod_{i \in I}H_{i}, \Delta,$ $\varepsilon, m, \eta),
(\pi_{i})_{i \in I}\bigl)$ be the product of this family in the
category $k$-BiAlg of bialgebras. The universal property of the
product yields an unique bialgebra map $S: B^{op,cop} \rightarrow
B$ such that the following diagram commutes for all $i \in I$:
\begin{equation}\eqlabel{antipod}
\xymatrix {& B \ar[r]^{\pi_{i}}  & {B_{i}} \\
& {B^{op,cop}}\ar[r]^{\pi_{i}} \ar[u]^{S} & {B_{i}^{op,cop}}
\ar[u]_{S_{i}}}
\end{equation}
Let $H$ be the sum of all subcoalgebras $C$ of the bialgebra $B$
such that : $S(c_{(1)})c_{(2)} = c_{(1)}S(c_{(2)}) = \eta \circ
\varepsilon(c)$ for all $c \in C$. The family of subcoalgebras $C$
satisfying the above property is nonempty by the same argument
used in the proof of \thref{1}. Moreover, it is easy to see that
\begin{equation}\eqlabel{S}
S(h_{(1)})h_{(2)} = h_{(1)}S(h_{(2)}) = \eta \circ \varepsilon(h)
\end{equation}
for all $h \in H$. We will prove that $H$ is a bialgebra and it
will follow by \equref{S} that $H$ is actually a Hopf algebra with
the antipode $S_{|H}$. First note that $\eta(k) = k1_{B} \subseteq H$.\\
Let $h, g \in H$. We then have:
\begin{eqnarray*}
S\bigl((hg)_{(1)}\bigl)(hg)_{(2)} &{=}& S(h_{(1)}g_{(1)})h_{(2)}g_{(2)}\\
&{=}& S(g_{(1)})S(h_{(1)})h_{(2)}g_{(2)}\\
&{=}& S(g_{(1)}) (\eta \circ \varepsilon)(h) g_{(2)}\\
&{=}& \bigl(\eta \circ \varepsilon(h) \bigl) \bigl(\eta \circ
\varepsilon(g) \bigl) = \eta \circ \varepsilon(hg)
\end{eqnarray*}
In the same way it can be proved that
$(hg)_{(1)}S\bigl((hg)_{(2)}\bigl) = \eta \circ \varepsilon(hg)$.
Thus $hg \in H$ and $H$ is indeed a bialgebra. In order to
conclude that $S_{|H}$ is an antipode for $H$ we need to prove
that $S(H) \subseteq H$. Let $h \in H$ ; we obtain:
\begin{eqnarray*}
S\bigl(S(h)_{(1)}\bigl)S(h)_{(2)} &{=}&
S\bigl(S(h_{(2)})\bigl)S(h_{(1)})\\
&{=}& S\bigl(h_{(1)}S(h_{(2)})\bigl)\\
&{=}& S\bigl( \eta \circ \varepsilon (h)\bigl)\\
&{=}& \eta \circ \varepsilon (h) = \eta \circ \varepsilon
\bigl(S(h)\bigl)
\end{eqnarray*}
A similar computation shows that we also have
$S(h)_{(1)}S\bigl(S(h)_{(2)}\bigl) = \eta \circ \varepsilon
\bigl(S(h)\bigl)$ for all $h \in H$.
Hence $H$ is a Hopf algebra with $S_{|H}$ as antipode.\\
To end the proof we still need to show that $\bigl((H, m, \eta,
\Delta, \varepsilon, S_{|H})\bigl), (q_{i})_{i \in I}\bigl)$ is
the product of the family $\bigl(H_{i}, m_{i}, \eta_{i},
\Delta_{i}, \varepsilon_{i}, S_{i}\bigl)_{i \in I}$ in the
category $k$-HopfAlg, where $q_{i} := \pi_{i} \circ j$ for all $i
\in I$
and $j: H \rightarrow B$ is the canonical inclusion.\\
Let $K$ be a Hopf algebra with antipode $S_{K}$ and $f_{i}: K
\rightarrow H_{i}$ be a family of Hopf algebra maps for all $i \in
I$. Since $B$ is the product in $k$-BiAlg of the above family of
Hopf algebras, there exist a unique morphism of bialgebras
$\overline{f}:K \rightarrow B$ such that the following diagram
commutes:
\begin{equation}\eqlabel{f}
\xymatrix {& K\ar[d]_{\overline{f}}\ar[dr]^{f_{i}}\\
& B \ar[r]^{\pi_{i}} & H_{i}}
\end{equation}
Using the fact that $f_{i}$ is a Hopf algebra map we have :
\begin{eqnarray*}
\pi_{i} \circ S \circ \overline{f} &\stackrel{\equref{antipod}}
{=}& S_{i} \circ \pi_{i} \circ \overline{f}\\
&\stackrel{\equref{f}}{=}& S_{i} \circ f_{i}\\
&{=}& f_{i} \circ S_{K}\\
&\stackrel{\equref{f}} {=}& \pi_{i} \circ \overline{f} \circ S_{K}
\end{eqnarray*}
for all $i \in I$. By the same argument used in the proof of
theorem \thref{2} it follows that:
\begin{equation}\eqlabel{A}
S \circ \overline{f} = \overline{f} \circ S_{K}
\end{equation}
Thus, for all $k \in K$ we have:
\begin{eqnarray*}
S(\overline{f}(k)_{(1)}) \overline{f}(k)_{(2)} &{=}&
S(\overline{f}(k_{(1)})) \overline{f}(k_{(2)})\\
&\stackrel{\equref{A}}{=}& \overline{f}\bigl(S_{K}(k_{(1)})\bigl)
\overline{f}(k_{(2)})\\
&{=}& \overline{f} \bigl(S_{K}(k_{(1)}) k_{(2)}\bigl)\\
&{=}& \overline{f}(k)
\end{eqnarray*}
Hence $\overline{f}(K) \subseteq H$. Thus, we obtained an unique
Hopf algebra map $\overline{f}: K \rightarrow H$ such that $q_{i}
\circ \overline{f} = f_{i}$ for all $i \in I$. Now, since the
forgetful functor $U:$ $k$-HopfAlg $\rightarrow$ $k$-BiAlg has a
left adjoint (see \cite{T}) it follows that, in particular, $U$
preserves products. That is, $H = B$ and the map $S$ obtained in
\equref{antipod} is actually an antipod for $B$. Thus, $\bigl((B,
m, \eta, \Delta, \varepsilon, S)\bigl), (\pi_{i})_{i \in I}\bigl)$
is the product of the family $\bigl(H_{i}, m_{i}, \eta_{i},
\Delta_{i}, \varepsilon_{i}, S_{i}\bigl)_{i \in I}$ in the
category $k$-HopfAlg.\\
Now let $f$, $g$ :$H \rightarrow K$ be two Hopf algebra morphisms
and $S: = \{h \in H | f(h) = g(h)\}$, which is just a $k$-subspace
of $H$. Let $D$ be the sum of all subcoalgebras of $H$ contained
in $S$. Again, the family of subcoalgebras of $H$ included in $S$
is not empty by the same argument used in \thref{1}. A simple
computation shows that $D$ is in fact a Hopf subalgebra of $H$.
Moreover, $(D, i)$ is the equalizer in the category $k$-HopfAlg of
the pair $(f,g)$ where $i: D \rightarrow H$ is the canonical
inclusion.
\end{proof}

\begin{remark}
Let $J$ be a small category, $F: J \rightarrow$$k$-HopfAlg be a
functor, $\bigl(\Pi_{j \in J} F(j),$ $(p_{j})_{j \in J} \bigl)$,
$\bigl(\Pi_{u \in Hom(J)} F(cod (u)), (p_{u})_{u \in Hom(J)}
\bigl)$ be the product in $k$-HopfAlg of the families
$\bigl(F(j)\bigl)_{j \in J}$, respectively $\bigl(F(cod
(u))\bigl)_{u \in Hom(J)}$ and $f, g : \Pi_{j \in J} F(j)
\rightarrow \Pi_{u \in Hom(J)} F(cod (u))$ be the unique Hopf
algebra maps such that $p_{u} \circ f = p_{cod (u)}$ and $p_{u}
\circ g = F(u) \circ p_{dom (u)}$ for all $u \in Hom(J)$. We
define
$$D = \{ x \in \Pi_{j \in J} F(j) | x_{(1)} \ot f (x_{(2)}) \ot
x_{(3)} = x_{(1)} \ot g (x_{(2)}) \ot x_{(3)} \}.$$ Then the pair
$\bigl(D, (\varphi_{j} = p_{j} \circ e)_{j \in J}\bigl)$ is the
limit of the functor $F$, where $e: D \rightarrow \Pi_{j \in J}
F(j)$ is the canonical inclusion.
\end{remark}

The key role in the construction of the product in the category of
coalgebras was played by the fact that the forgetful functor from
the category of coalgebras to the category of vector spaces has a
right adjoint. It is therefore natural to ask if the conclusion
remains true for the category of $R$-corings (\cite{BW}). Let $R$
be a ring, $R$-Corings be the category of $R$-corings and
${}_R{\mathcal M}_R$ be the category of $R$-bimodules.

\textbf{Problem:} \textit{Does there exist a right adjoint for the
forgetful functor $F : R - Corings \to {}_R{\mathcal M}_R$}?

\section*{Acknowledgement}

The author wishes to thank Professor Gigel Militaru, who suggested
the problems studied here, for his great support and for the
useful comments from which this manuscript has benefitted.

\end{document}